\documentclass[11pt,a4paper,english,reqno]{amsart}
\usepackage[T1]{fontenc}
\usepackage[ansinew]{inputenc}
\usepackage{amsthm}
\usepackage{amstext}
\usepackage{amssymb}
\usepackage{esint}
\usepackage[left=3cm,right=3cm,top=3cm,bottom=3cm]{geometry}
\usepackage{amsfonts}
\usepackage{amsthm}
\usepackage{amsfonts}
\usepackage{babel}

\makeatletter

\curraddr[P. A. Razafimandimby]{Mathematics Section\\ The Abdus Salam International Center for Theoretical Physics\\
Strada Costiera, 11, I - 34151 Trieste, Italy}
\email[P.A. Razafimandimby]{prazafim@ictp.it,paul.razafim@gmail.com}
\email[M. Sango]{sango1767@gmail.com}
\theoremstyle{plain}
\newtheorem{lem}{Lemma}[section]
\newtheorem{thm}[lem]{Theorem}
\newtheorem{prop}[lem]{Proposition}
\theoremstyle{definition}
\newtheorem{Def}[lem]{Definition}
\newtheorem{Rem}[lem]{Remark}

\DeclareMathOperator{\Div}{div}
\DeclareMathOperator{\curl}{curl}

\thanks{}
\makeatother

\begin{document}
\title[Convergence of solution of stochastic equations of second grade fluids]{Convergence of a sequence of solutions of the stochastic
two-dimensional equations of second grade fluids  }
\author[P. A. Razafimandimby]{Paul Andr\'e RAZAFIMANDIMBY}
\author[M. Sango]{Mamadou SANGO}
\address[P. A. Razafimandimby and M. Sango]{Department of Mathematics and
Applied Mathematics\\
University of Pretoria\\
Lynwood Road, Pretoria 0002, South Africa}

\begin{abstract}
We study the limit of the stochastic model for two dimensional second grade
fluids subjected to the periodic boundary conditions as the stress modulus
tends to zero. We show that under suitable conditions on the data the whole
sequence of strong probabilistic solutions $\left( u^{\alpha }\right) $ of
the stochastic second grade fluid converges to the unique strong
probabilistic solution of the stochastic Navier-Stokes equations.
\end{abstract}

\maketitle

\section{Introduction}

Let $D=[0,L]^{2}\subset \mathbb{R}^{2}$, $L>0$, be a periodic square, $T>0$
a fixed time. Let $\alpha $ denote a sequence of positive numbers $\left( \alpha _{n}\right)
_{n\in \mathbb{N}}$ which converges to zero as $n$ converges to $\infty $;
we shall express this by just writing $\alpha \rightarrow 0$.\ We consider a
complete probability space $(\Omega ,\mathcal{F},P)$ endowed with the
filtration $\mathcal{F}^{t}$, $0\leq t\leq T$, which is the $\sigma $-field
generated by a given $\mathbb{R}^{m}$-valued standard Wiener process $\{%
\mathcal{W}(s),0\leq s\leq T\}$ and the null sets of $\mathcal{F}$. In this
paper we investigate the behavior of the sequence $\left( u^{\alpha }\right)
$ of strong probabilistic solutions of the following problems:
\begin{equation}
\begin{cases}
d(u^{\alpha }-\alpha \Delta u^{\alpha })+(-\nu \Delta u^{\alpha }+\curl%
(u^{\alpha }-\alpha \Delta u^{\alpha })\times u^{\alpha }+\nabla \mathfrak{P}%
)dt=F\,dt+Gd\mathcal{W} \\
\text{in}\,\,\Omega \times (0,T]\times D, \\
\Div u^{\alpha }=0\,\,\text{in}\,\,\Omega \times (0,T]\times D, \\
\int_{D}u^{\alpha }dx=0\,\,\text{in}\,\,\Omega \times (0,T], \\
u^{\alpha }(0)=u_{0}\,\,\text{in}\,\,\Omega \times D,%
\end{cases}
\label{secondgrade}
\end{equation}%
when $\alpha \rightarrow 0$. The system (\ref{secondgrade}), which is to be
understood in the sense of distributions, is the equations of motion for an
incompressible second grade fluid driven by random external forces. Here $%
u^{\alpha }$ is the velocity of the fluid, $\mathfrak{P}$ is a modified
pressure given by
\begin{equation*}
\mathfrak{P}=-\tilde{p}-(1/2)|u^{\alpha }|_{\mathbb{R}^{2}}^{2}+\alpha
u^{\alpha }.\Delta u^{\alpha }+(\alpha /4)\text{tr}(\nabla u^{\alpha
}+(\nabla u^{\alpha })^{t}).
\end{equation*}%
Throughout we assume that (\ref{secondgrade}) is subject to the periodic
boundary condition. We refer to \cite{noll} and \cite{dunn} for further
reading on fluid of complexity two and on second grade fluids. The interest
in the investigation of mathematical and physical problems related to second
grade fluids arises from the fact that they describe a large class of
Non-Newtonian fluids such as dilute polymeric solutions (solution of swollen
gel or oil polyols ), industrial fluids (oils,...), slurry flows; just to
cite a few. Second grade fluids are also connected to Turbulence Theory.
Indeed the discussion on the relation between Non-Newtonian fluids,
especially fluids of differential type, and Turbulence Theory started with
the work of Rivlin \cite{RIVLIN}. It was rediscovered recently (see, for
example, \cite{TITI3} and \cite{CHEN} ) that the flow of second grade fluids
can be used as a basis for a turbulence closure model.

In the deterministic case, i.e when $G(t,x)\equiv 0$, existence and
uniqueness results are given in \cite{ouazar}, \cite{girault} for instance.
It is known from \cite{IFTIMIE} that under general assumption on the data
the weak solution (in the partial differential equations sense) of second
grade fluids equations converges weakly to the weak solution of the
Navier-Stokes equations. We also refer to \cite{busioc1} for interesting
discussions related to their relationship with other fluid models. Although
there are lots of papers dealing with stochastic partial differential
equations and hydrodynamics (see, for instance, \cite{bensoussan2}, \cite%
{bensoussan}, \cite{breckner},\cite{Brzezniak3}, \cite{Brzezniak4}, \cite%
{Caraballo}, \cite{Caraballo1}, \cite{Caraballo3}, \cite{DEUGOUE1}, \cite%
{DEUGOUE2} ,\cite{GYONGY},\cite{Menaldi}, \cite{PAUL},\cite{sango1}, \cite%
{sango2},\cite{sango}), there are only few known results for the stochastic
version of second grade fluids. The existence of weak probabilistic (or
martingale) solution was recently proved in \cite{PAUL1}. In this paper, we
show that we can construct a sequence $\left( u^{\alpha _{j}}\right) $ of
strong stochastic solutions of (\ref{secondgrade}) that converges in a
certain sense (see Theorem \ref{Mainthm} and Remark \ref{strong}) to the
stochastic weak solution of the stochastic Navier-Stokes equations (SNSE) as
$\alpha _{j}\rightarrow 0$. This result was first established in \cite{PAUL3}.
We also prove
that the whole sequence $\left( u^{\alpha }\right) $ converges in
probability to the unique stochastic strong solution of the stochastic
Navier-Stokes equations in the topology of $L^{2}(0,T;\mathbb{H})$ as $%
\alpha \rightarrow 0$ (see Theorem \ref{CONVPROB}). Our proof, which is
inspired by the papers \cite{bensoussan} and \cite{IFTIMIE}, relies on
deriving estimates (independent of $\alpha $) for the velocity $u^{\alpha }$
in Sobolev $\mathbb{H}^{1}$ norm. Unfortunately, \eqref{secondgrade}
contains a very highly nonlinear term (see the $\curl$-term) and using only
the information on the $\mathbb{H}^{1}$ norm of $u^{\alpha }$ is not
sufficient to pass to the limit in this term. To overcome this difficulty,
we proved an interesting and technical mean-type estimate (see Lemma \ref%
{finite-diff}) which allows to use the deep compactness results of Prokhorov
and Skorokhod. This approach is a probabilistic refinement of the idea in
\cite{IFTIMIE}. The convergence of the whole sequence to the strong
probabilistic solution of the Stochastic Navier-Stokes equations is obtained
by using a technical lemma (Lemma \ref{CONVPROB1} ) which originated in \cite%
{GYONGY}. The present paper generalizes the deep result obtained by Iftimie
in \cite{IFTIMIE}. Our work also emphasizes the theory of Rivlin in \cite%
{RIVLIN}.

The layout of this paper is as follows. In addition to the current
introduction this article consists of four other sections.  In Section 2 we
give some notations, necessary backgrounds of probabilistic or analytical
nature. We formulate the hypotheses relevant for the paper and our main
results in Section 3. The fourth section is devoted to the proof of the
first main result.  Finally, we prove in the last section that the whole
sequence of the strong probabilistic solution for the stochastic model for
second grade fluids  converges to that of the stochastic Navier-Stokes
equations in dimension two.

\section{Preliminaries-Notations}

For a Banach space $X$ we denote by $\mathbb{X}$ the space of $\mathbb{R}^{2}
$-valued functions such that each component is an element of $X$. We denote
by $\mathbb{H}_{0}^{1}(D)$ the space of functions $u$ that belong to the
Sobolev space of periodic functions $\mathbb{H}^{1}(D)$ and satisfying
\begin{equation*}
\int_D u(x)dx=0.
\end{equation*}
We also introduce the spaces
\begin{align*}
\mathcal{V} & =\left\{ u\in\mathcal{C}_{per}^{\infty}(D):\Div u=0\text{ and }%
\int_{D}udx=0\right\} \\
\mathbb{V} & =\,\,\text{closure of $\mathcal{V}$\text{ in } }\,\,\mathbb{H}_{0}^{1}(D)
\\
\mathbb{H} & =\,\,\text{closure of $\mathcal{V}$ \text{ in } }\,\,\mathbb{L}^{2}(D),
\end{align*}
where $\mathcal{C}_{per}^{\infty}(D)$ denotes the space of infinitely
differentiable periodic function with period $L$.\newline
We denote by $(\cdot,\cdot)$ and $|\cdot|$ the inner product and the norm
induced by the inner product and the norm in $\mathbb{L}^{2}(D)$ on $\mathbb{%
H}$, respectively. Thanks to Poincaré's inequality, we can endow $\mathbb{V}$
with the gradient scalar product (resp. the norm) $\big((.,.)\big)$ (resp. $%
||.||$). In the space $\mathbb{V}$, the latter norm is equivalent to the
norm generated by the following scalar product
\begin{equation*}
(u,v)_{\mathbb{V}}=(u,v)+\alpha((u,v)),\,\,\text{for any}\,\, u\,\text{and}%
\, v\in\mathbb{V}.
\end{equation*}
More precisely we have
\begin{equation}
(\mathcal{P}+\alpha)^{-1}{|}{v}{|_{\mathbb{V}}}^{2}\le||v||^{2}\le(%
\alpha)^{-1}{|}{v}{|_{\mathbb{V}}}^{2},\,\,\text{for any}\,\, v\in\mathbb{V}.
\label{equiv}
\end{equation}
We also introduce the following space
\begin{equation*}
\mathbb{W}=\left\{ u:\Div{u}=0,\int_{D}udx=0,\text{ and }\curl{(u-\alpha%
\Delta u)}\in L^{2}(D)\right\} .
\end{equation*}
We provide this space with the norm $|.|_{\mathbb{W}}$ generated by the
scalar product
\begin{equation*}
(u,v)_{\mathbb{W}}=(\curl(u-\alpha\Delta u),\curl(v-\alpha\Delta v)).
\end{equation*}
This norm is equivalent to the usual $\mathbb{H}^{3}(D)$-norm on $\mathbb{W}$%
. For any Banach space $X$, $p,r\geq1$, we set $L^{p,r}(0,T,\Omega,X)=L^{p}(%
\Omega,\mathcal{F},P;L^{r}(0,T;X)$. Next we give some results on which most
of proofs in forthcoming sections rely.

\begin{thm}
{[} see \cite{simon}{]} \label{comsimon} Let $X,B,Y$ three Banach spaces
such that the following embedding are continuous
\begin{equation*}
X\subset B\subset Y.
\end{equation*}
Moreover, assume that the embedding $X\subset B$ is compact, then the set $%
\mathfrak{F}$ consisting of functions $v\in L^{q}(0,T;B)\cap
L_{loc}^{1}(0,T;X)$, $1\le q\le\infty$ such that
\begin{equation*}
\sup_{0\le h\le1}\int_{t_{1}}^{t_{2}}|v(t+h)-v(t)|_{Y}dt\rightarrow0,\,\text{%
as}\, h\rightarrow0,
\end{equation*}
for any $0<t_{1}<t_{2}<T$ is compact in $L^{p}(0,T;B)$ for any $p<q$.
\end{thm}

We also need the following product formulas, we refer to \cite{Chemin} for
their proof in the case of the whole space (see \cite{Gallagher} for the
case of periodic condition).

\begin{thm}
\label{product-formula} Let $D$ be a $n$-dimensional periodic box and let $%
\beta,\gamma\in\mathbb{R}$ such that $\beta+\gamma>0$, $\beta<\frac{n}{2}$, $%
\gamma<\frac{n}{2}$. If $u\in\mathbb{H}^{\gamma}(D)$ and $v\in\mathbb{H}%
^{\beta}(D)$, then there exists a positive constant $C$ such that
\begin{equation*}
|uv|_{\mathbb{H}^{\gamma+\beta-\frac{n}{2}}}\le C|u|_{\mathbb{H}%
^{\gamma}}|v|_{\mathbb{H}^{\beta}}.
\end{equation*}
If $|\gamma|<\frac{n}{2}$, then
\begin{equation}
|uv|_{\mathbb{H}^{-\frac{n}{2}-\varepsilon}}\le C^{\prime}|u|_{\mathbb{H}%
^{\gamma}}|v|_{\mathbb{H}^{-\gamma}},
\end{equation}
for any $u\in\mathbb{H}^{\gamma}(D)$, $v\in\mathbb{H}^{-\gamma}(D)$ and $%
\varepsilon>0$ .
\end{thm}

\section{Hypotheses and the main results}

In this section we will recall briefly the previous results on the problem (%
\ref{secondgrade}) and will formulate our main results.

\subsection{The hypotheses}

We assume that

\begin{enumerate}
\item[(I)] $F=F(t,x)$ is a $\mathbb{V}$-valued function defined on $%
[0,T]\times D$ such that the following holds for any $2\le p<\infty$
\begin{equation*}
\int_{0}^{T}|F(t,x)|_{\mathbb{V}}^{p}<\infty.
\end{equation*}

\item[(II)] $G=G(t,x)$ is a $\mathbb{V}^{\otimes m}$-valued function defined
on $[0,T]\times D$ such that the following holds
\begin{equation*}
\int_{0}^{T}|G(t,x)|_{\mathbb{V}^{\otimes m}}^{p}<\infty,
\end{equation*}
for any $2\le p<\infty$.

\item[(III)] We further assume that $u_{0}\subset\mathbb{V}\cap\mathbb{H}^{3}
$ is nonrandom and that there exists a positive constant $C$ independent of $%
\alpha$ such that $|u_{0}|_{\mathbb{V}}<C$. Suppose also that $\nu>0$.
\end{enumerate}

We continue with the definition of the concept of the strong probabilistic
solution for the problem (\ref{secondgrade}).

\begin{Def}
\label{defstrong}  By a strong probabilistic solution of the system %
\eqref{secondgrade}, we mean a stochastic process $u^{\alpha}$ such that

\begin{enumerate}
\item $u^{\alpha}\in L^{p}(\Omega,\mathcal{F},P;L^{\infty}(0,T;\mathbb{W}%
))\cap L^{p}(\Omega,\mathcal{F},P;L^{\infty}(0,T;\mathbb{V}))$ with $2\le
p<\infty,$

\item For almost all $t$, $u^{\alpha}(t)$ is $\mathcal{F}^{t}$-measurable,

\item $P$-a.s the following integral equation holds
\begin{equation*}
\begin{split}
& (u^{\alpha}(t)-u^{\alpha}(0),\phi)_{\mathbb{V}}+\int_{0}^{t}\left[%
\nu((u^{\alpha},\phi))+(\curl(u^{\alpha}(s)-\alpha\Delta
u^{\alpha}(s))\times u^{\alpha}(s),\phi)\right]ds \\
& =\int_{0}^{t}(F,\phi)ds+\int_{0}^{t}(G,\phi)d\mathcal{W}(s)
\end{split}
\end{equation*}
for any $t\in(0,T]$ and $\phi\in\mathcal{V}$.
\end{enumerate}
\end{Def}

\begin{Rem}
\label{REMARK1}  In the above definition the quantity $\int_{0}^{t}(G,\phi)d%
\mathcal{W}(s)$ should be understood as
\begin{equation*}
\int_{0}^{t}(G,\phi)d\mathcal{W}(s)=\sum_{k=1}^{m}\int_{0}^{t}(G_{k},\phi)d%
\mathcal{W}_{k}(s),
\end{equation*}
where $G_{k}$ and $\mathcal{W}_{k}$ denote the $k$-th component of $G$ and $%
\mathcal{W}$, respectively.
\end{Rem}

We recall now two results from the paper \cite{paul-sango2}.

\begin{thm}
Under the assumptions (I), (II) and (III) the problem \eqref{secondgrade}
has a solution in the sense of the above definition. Moreover, almost surely
the paths of the solution are $\mathbb{W}$-valued weakly continuous.

Let $u_{1}^{\alpha}$ and $u_{2}^{\alpha}$ be two strong probabilistic
solutions of the problem (\ref{secondgrade}) defined on the stochastic basis
$(\Omega,\mathcal{F},\mathcal{F}^{t},P)$. If we set $U^{\alpha}=u_{1}^{%
\alpha}-u_{2}^{\alpha}$, then we have $U^{\alpha}=0$ almost surely.
\end{thm}

\subsection{Statement of the main theorems}

Before we proceed to the statement of our main theorem we introduce the SNSE
\begin{equation}
\begin{cases}
dv+(-\nu \Delta v+(v.\nabla v)+\nabla \mathcal{P})dt=Fdt+Gd\bar{\mathcal{W}}
\\
\text{in}\,\,\bar{\Omega}\times (0,T]\times D, \\
\Div v=0\,\,\text{in}\,\,\bar{\Omega}\times (0,T]\times D, \\
\int_{D}vdx=0\,\,\text{in}\,\,\bar{\Omega}\times (0,T], \\
v(0)=u_{0}\,\,\text{in}\,\,\bar{\Omega}\times D.%
\end{cases}
\label{SNSE}
\end{equation}%
and recall the concept of a weak probabilistic solution of the problem.

\begin{Def}
\label{defweaknc} By a weak probabilistic solution of (\ref{SNSE}), we mean
a system
\begin{equation*}
(\bar{\Omega},\bar{\mathcal{F}},\bar{P},\bar{\mathcal{F}}^{t},\bar{\mathcal{W%
}},v),
\end{equation*}%
where

\begin{enumerate}
\item $(\bar{\Omega},\bar{\mathcal{F}},\bar{P})$ is a complete probability
space, $\bar{\mathcal{F}}^{t}$ is a filtration on $(\bar{\Omega},\bar{%
\mathcal{F}},\bar{P})$,

\item $\bar{\mathcal{W}}(t)$ is an $m$-dimensional $\bar{\mathcal{F}}^{t}$%
-standard Wiener process,

\item $v(t)\in L^{p}(\bar{\Omega},\bar{\mathcal{F}},\bar{P};L^{2}(0,T;%
\mathbb{V}))\cap L^{p}(\bar{\Omega},\bar{\mathcal{F}},\bar{P};L^{\infty}(0,T;%
\mathbb{H}))$, $\forall$\,\,\,\, $2\le p<\infty,$

\item For almost all $t$, $v(t)$ is $\bar{\mathcal{F}^{t}}$-measurable,

\item P-a.s the following integral equation holds
\begin{equation}
\begin{split}
& (v(t)-v(0),\phi)+\int_{0}^{t}\left[\nu((v,\phi))+<v.\nabla v,\phi>\right]ds
\\
& =\int_{0}^{t}(F,\phi)ds+\int_{0}^{t}(G,\phi)d\bar{\mathcal{W}}(s)
\end{split}
\label{solve}
\end{equation}
for any $t\in(0,T]$ and $\phi\in\mathcal{V}$.
\end{enumerate}
\end{Def}

In this article we prove the following result.

\begin{thm}
\label{Mainthm} Under the hypotheses (I)-(III) there exist a probability
space $(\bar{\Omega},\bar{\mathcal{F}},\bar{P})$, a family of probability
measures $(\Pi ^{\alpha _{j}})$, a probability measure $\Pi $, and
stochastic processes $(\mathcal{W}^{\alpha _{j}},u^{\alpha _{j}})$, $(\bar{%
\mathcal{W}},v)$ such that the law of $(\mathcal{W}^{\alpha _{j}},u^{\alpha
_{j}})$ (resp. $(\bar{\mathcal{W}},v)$) is $\Pi ^{\alpha _{j}}$ (resp. $\Pi $%
) and $\mathcal{W}^{\alpha _{j}}\rightarrow \bar{\mathcal{W}}$ uniformly $%
\bar{P}$-a.s. when $j\rightarrow \infty $ ($\alpha _{j}\rightarrow 0$). The
pair $(\mathcal{W}^{\alpha _{j}},u^{\alpha _{j}})$ satisfies $\bar{P}-$a.s. %
\eqref{secondgrade} in the sense of distribution and as $j\rightarrow \infty
$ ($\alpha _{j}\rightarrow 0$)
\begin{align}
u^{\alpha _{j}}& \rightharpoonup v,\,\text{weakly in}\,\,L^{p}(\bar{\Omega},%
\bar{\mathcal{F}},\bar{P};L^{2}(0,T;\mathbb{V})), \\
u^{\alpha _{j}}& \rightharpoonup v\,\,\text{weakly-* in}\,\,L^{p}(\bar{\Omega%
},\bar{\mathcal{F}},\bar{P};L^{\infty }(0,T;\mathbb{H})),
\end{align}%
for all $2\leq p<\infty $ and $(\bar{\Omega},\bar{\mathcal{F}},\bar{P},v,%
\bar{\mathcal{W}})$ is a weak probabilistic solution of (\ref{SNSE}).
\end{thm}

\begin{Rem}
\label{strong} Since we are in 2-D then it is known that under our
hypotheses (I)-(III) the problem (\ref{SNSE}) has a strong probabilistic
solution which is unique, see for example \cite{Menaldi}. This implies that
the process $v$ of the above theorem is a strong probabilistic solution of
the Stochastic Navier Stokes Equations (\ref{SNSE}).
\end{Rem}

The convergence of the whole sequence $\left( u^{\alpha }\right) $ to the
(either strong or weak probabilistic) solution of the stochastic
Navier-Stokes equations could not be proved in \cite{PAUL3}. The second main
goal of this work is to prove the following convergence result.

\begin{thm}
\label{CONVPROB} For the given stochastic system $(\Omega,\mathcal{F},P),%
\mathcal{F}^{t},\mathcal{W}$ defined in Introduction, we have that the whole
sequence $u^{\alpha}$ converges in probability to $v$ in the topology of $%
L^{2}(0,T;\mathbb{H})$, i.e $||u^{\alpha}-v||_{L^{2}(0,T;\mathbb{H}}$
converges to zero in probability. Here $v$ is the unique strong
probabilistic solution of the stochastic Navier-Stokes equations.
\end{thm}

\section{Proof of Theorem \protect\ref{Mainthm}}

This section is devoted to the proof of our first main result.

\subsection{Uniform a priori estimates}

In this subsection we derive some estimates uniform in $\alpha $. These
inequalities do not follow from previous works (see \cite{PAUL}) which
explode when $\alpha \rightarrow 0$. Throughout $C$ denotes unessential
positive constant independent of $\alpha $, and which may change from one
line to the next. Since we $\alpha \rightarrow 0$ we may assume that $\alpha
=\left( \alpha _{n}\right) \subset \lbrack 0,1)$, at least for large $n$ .

\begin{lem}
\label{lem5} For $\alpha \in (0,1)$ we have
\begin{align}
E\sup_{0\leq s\leq T}(|u^{\alpha }(s)|^{2}+\alpha ||u^{\alpha
}(s)||^{2})+E\int_{0}^{T}||u^{\alpha }(s)||^{2}ds& \leq C,  \label{lem51} \\
E\sup_{0\leq s\leq T}(|u^{\alpha }(s)|^{2}+\alpha ||u^{\alpha }(s)||^{2})^{%
\frac{p}{2}}+E\left( \int_{0}^{T}||u^{\alpha }(s)||^{2}ds\right) ^{\frac{p}{2%
}}& \leq C,  \label{lem52}
\end{align}%
for any $2\leq p<\infty $.
\end{lem}

Before we prove this result it is important to make the following remark.

\begin{Rem}
\label{solo} We recall that the continuous linear operator $(I+\alpha A)^{-1}
$, where $A$ is the usual Stokes operator, establishes a bijective
correspondence between the spaces $\mathbb{H}^{l}(D)\cap\mathbb{V}$ (resp. $%
\mathbb{H}$) and $\mathbb{H}^{l+2}(D)\cap\mathbb{V}$, $l>1$ (resp. $l=0$).
Furthermore for any $w\in\mathbb{V}$, and $f\in\mathbb{H}^{l}(D)$, $l\ge0$,
\begin{align}
((I+\alpha A)^{-1}f,w)_{\mathbb{V}} & =(f,w),  \label{relation} \\
|(I+\alpha A)^{-1}f|_{\mathbb{V}} & \le C|f|  \label{related}
\end{align}
\end{Rem}

\begin{proof}
We start the proof of our lemma by proving (\ref{lem51}). Since $u^{\alpha }$
is a solution of (\ref{secondgrade}) then
\begin{equation*}
du^{\alpha }+(I+\alpha A)^{-1}Au^{\alpha }dt+(I+\alpha A)^{-1}\widehat{B}%
(u^{\alpha },u^{\alpha })dt=(I+\alpha A)^{-1}Fdt+(I+\alpha A)^{-1}Gd\mathcal{%
W},
\end{equation*}%
holds $P$-a.s. for any $t\in \lbrack 0,T]$. Here we have set
\begin{equation*}
\widehat{B}(u^{\alpha },u^{\alpha })=\curl(u^{\alpha }-\alpha \Delta
u^{\alpha })\times u^{\alpha }.
\end{equation*}%
For the rest of this section and the paper we write
\begin{align*}
(I+\alpha A)^{-1}F& =\widehat{F}, \\
(I+\alpha A)^{-1}G& =\widehat{G}.
\end{align*}%
Ito's formula implies that
\begin{equation*}
\begin{split}
& d|u^{\alpha }|_{\mathbb{V}}^{2}+2((I+\alpha A)^{-1}Au^{\alpha },u^{\alpha
})_{\mathbb{V}}dt+2((I+\alpha A)^{-1}\widehat{B}(u^{\alpha },u^{\alpha
}),u^{\alpha })_{\mathbb{V}}dt \\
& =(\widehat{F},u^{\alpha })_{\mathbb{V}}dt+|\widehat{G}|_{\mathbb{V}%
^{\otimes m}}^{2}dt+2(\widehat{G},u^{\alpha })_{\mathbb{V}}d\mathcal{W}.
\end{split}%
\end{equation*}%
By the relationship (\ref{relation}) in the above remark and the equation
\begin{equation*}
(\widehat{B}(u^{\alpha },u^{\alpha }),u^{\alpha })=0,
\end{equation*}%
we obtain that
\begin{equation*}
d|u^{\alpha }|_{\mathbb{V}}^{2}+2||u^{\alpha }||^{2}dt=2(F,u^{\alpha })dt+|%
\widehat{G}|_{\mathbb{V}^{\otimes m}}^{2}dt+2(\widehat{G},u^{\alpha })_{%
\mathbb{V}}d\mathcal{W}.
\end{equation*}%
This relation combined with Cauchy-Schwarz's inequality and (\ref{related})
imply that
\begin{equation*}
d|u^{\alpha }|_{\mathbb{V}}^{2}+2||u^{\alpha }||^{2}dt\leq
(|F|^{2}+|u^{\alpha }|^{2})dt+|\widehat{G}|_{\mathbb{V}^{\otimes m}}^{2}dt+2(%
\widehat{G},u^{\alpha })_{\mathbb{V}}d\mathcal{W}.
\end{equation*}%
Recalling the definition of $|.|_{\mathbb{V}}^{2}$ we deduce that
\begin{equation}
d|u^{\alpha }|_{\mathbb{V}}^{2}+2||u^{\alpha }||^{2}dt\leq |F|_{\mathbb{V}%
}^{2}+|u^{\alpha }|_{\mathbb{V}}^{2}dt+|\widehat{G}|_{\mathbb{V}^{\otimes
m}}^{2}dt+2(\widehat{G},u^{\alpha })_{\mathbb{V}}d\mathcal{W}.  \label{lem53}
\end{equation}%
Taking the $\sup $ over $0\leq s\leq t$, $t\in \lbrack 0,T]$ and passing to
the mathematical expectation yield
\begin{equation*}
E\sup_{0\leq s\leq t}|u^{\alpha }|_{\mathbb{V}}^{2}+2E\int_{0}^{t}||u^{%
\alpha }||^{2}ds\leq C+E\int_{0}^{t}|u^{\alpha }|_{\mathbb{V}%
}^{2}ds+2E\sup_{0\leq s\leq t}\left\vert \int_{0}^{s}(\widehat{G},u^{\alpha
})_{\mathbb{V}}d\mathcal{W}\right\vert ,
\end{equation*}%
where the assumptions on $F$ and $G$ were used. Burkhölder-Davis-Gundy's
inequality implies
\begin{equation*}
E\sup_{0\leq s\leq t}|u^{\alpha }|_{\mathbb{V}}^{2}+2E\int_{0}^{t}||u^{%
\alpha }||^{2}ds\leq C+E\int_{0}^{t}|u^{\alpha }|_{\mathbb{V}%
}^{2}ds+6E\left( \int_{0}^{t}(\widehat{G},u^{\alpha })_{\mathbb{V}%
}^{2}ds\right) ^{\frac{1}{2}}.
\end{equation*}%
Cauchy's inequality implies
\begin{equation*}
E\sup_{0\leq s\leq t}|u^{\alpha }|_{\mathbb{V}}^{2}+2E\int_{0}^{t}||u^{%
\alpha }||^{2}ds\leq C+E\int_{0}^{t}|u^{\alpha }|_{\mathbb{V}}^{2}ds+\frac{1%
}{2}E\sup_{0\leq s\leq t}|u^{\alpha }(s)|_{\mathbb{V}}^{2}+CE\int_{0}^{t}|%
\widehat{G}|_{\mathbb{V}^{\otimes m}}^{2}ds,
\end{equation*}%
or
\begin{equation*}
E\sup_{0\leq s\leq t}|u^{\alpha }|_{\mathbb{V}}^{2}+4E\int_{0}^{t}||u^{%
\alpha }||^{2}ds\leq C+CE\int_{0}^{t}|u^{\alpha }|_{\mathbb{V}}^{2}ds.
\end{equation*}%
Here we have used (\ref{related}) and the assumption on $G$. It follows from
Gronwall's inequality that
\begin{equation*}
E\sup_{0\leq s\leq t}|u^{\alpha }|_{\mathbb{V}}^{2}+2E\int_{0}^{t}||u^{%
\alpha }||^{2}ds<C,
\end{equation*}%
for any $t\in \lbrack 0,T]$. This completes the proof of (\ref{lem51}).

We continue with the proof of (\ref{lem52}). For $2\leq p<\infty $ and $t\in
\lbrack 0,T]$ the following holds:
\begin{equation*}
\begin{split}
|u^{\alpha }|_{\mathbb{V}}^{p}+p\int_{0}^{t}|u^{\alpha }|_{\mathbb{V}%
}^{p-2}||u^{\alpha }||^{2}ds& =|u_{0}|_{\mathbb{V}}^{p}+p\int_{0}^{t}|u^{%
\alpha }|_{\mathbb{V}}^{p-2}(\widehat{F},u^{\alpha })ds+\frac{p}{2}%
\int_{0}^{t}|u^{\alpha }|_{\mathbb{V}}^{p-2}|\widehat{G}|_{\mathbb{V}}^{2}ds
\\
& +\frac{(p-2)p}{2}\int_{0}^{t}|u^{\alpha }|_{\mathbb{V}}^{p-4}(\widehat{G}%
,u^{\alpha })_{\mathbb{V}}ds+p\int_{0}^{t}|u^{\alpha }|_{\mathbb{V}}^{p-2}(%
\widehat{G},u^{\alpha })_{\mathbb{V}}d\mathcal{W}.
\end{split}%
\end{equation*}%
Owing to (\ref{related}) and the above estimate we see that
\begin{equation*}
\begin{split}
|u^{\alpha }|_{\mathbb{V}}^{p}& \leq |u_{0}|_{\mathbb{V}}^{p}+p%
\int_{0}^{t}|u^{\alpha }|_{\mathbb{V}}^{p-1}|F|ds+\frac{p}{2}%
C\int_{0}^{t}|u^{\alpha }|_{\mathbb{V}}^{p-2}|G|^{2}ds+\frac{(p-2)p}{2}%
C\int_{0}^{t}|u^{\alpha }|_{\mathbb{V}}^{p-2}|G|^{2}ds \\
& +p\int_{0}^{t}|u^{\alpha }|_{\mathbb{V}}^{p-2}(\widehat{G},u^{\alpha })_{%
\mathbb{V}}d\mathcal{W}.
\end{split}%
\end{equation*}%
We derive from this by using Young's inequality that
\begin{equation*}
|u^{\alpha }|_{\mathbb{V}}^{p}\leq |u_{0}|_{\mathbb{V}}^{p}+C%
\int_{0}^{t}|u^{\alpha }|_{\mathbb{V}}^{p}ds+C\int_{0}^{t}|F|^{p}ds+C%
\int_{0}^{t}|G|^{p}ds+p\int_{0}^{t}|u^{\alpha }|_{\mathbb{V}}^{p-2}(\widehat{%
G},u^{\alpha })_{\mathbb{V}}d\mathcal{W}.
\end{equation*}%
Taking the $\sup $ over $0\leq s\leq t$, passing to the mathematical
expectation and using the assumptions on $F$ and $G$ imply
\begin{equation*}
E\sup_{0\leq s\leq t}|u^{\alpha }|_{\mathbb{V}}^{p}\leq |u_{0}|_{\mathbb{V}%
}^{p}+CE\int_{0}^{t}|u^{\alpha }|_{\mathbb{V}}^{p}ds+pE\sup_{0\leq s\leq
t}\left\vert \int_{0}^{s}|u^{\alpha }|_{\mathbb{V}}^{p-2}(\widehat{G}%
,u^{\alpha })_{\mathbb{V}}d\mathcal{W}\right\vert .
\end{equation*}%
Invoking the Martingale inequality yields
\begin{equation*}
E\sup_{0\leq s\leq t}|u^{\alpha }|_{\mathbb{V}}^{p}\leq
CE\int_{0}^{t}|u^{\alpha }|_{\mathbb{V}}^{p}ds+pE\left(
\int_{0}^{s}t|u^{\alpha }|_{\mathbb{V}}^{2p-2}|\widehat{G}|_{\mathbb{V}%
}^{2}ds\right) ^{\frac{1}{2}}.
\end{equation*}%
We infer from this estimate, Young's inequality, (\ref{related}) along with
the assumption on $G$ and Gronwall's inequality that
\begin{equation}
E\sup_{0\leq s\leq t}|u^{\alpha }(s)|_{\mathbb{V}}^{p}<C,  \label{lem54}
\end{equation}%
for any $t\in \lbrack 0,T]$ and $2\leq p<\infty $. We deduce from (\ref%
{lem53}) with the help of this last estimate that
\begin{equation*}
E\left( \int_{0}^{t}||u^{\alpha }||^{2}ds\right) ^{\frac{p}{2}}\leq
C+CE\left\vert \int_{0}^{t}(\widehat{G},u^{\alpha })_{\mathbb{V}}d\mathcal{W}%
\right\vert ^{\frac{p}{2}}.
\end{equation*}%
We obtain from this with the help of the Martingale inequality and (\ref%
{lem54}) that
\begin{equation*}
E\left( \int_{0}^{t}||u^{\alpha }||^{2}ds\right) ^{\frac{p}{2}}\leq C.
\end{equation*}%
And this completes the proof of (\ref{lem52}), hence the lemma.
\end{proof}

\begin{Rem}
For $1\leq p<\infty $, the following estimates are valid
\begin{equation}
E\sup_{0\leq s\leq T}(|u^{\alpha }(s)|^{2}+\alpha ||u^{\alpha }(s)||^{2})^{%
\frac{p}{2}}+E\left( \int_{0}^{T}||u^{\alpha }(s)||^{2}ds\right) ^{\frac{p}{2%
}}<C,  \label{lem52vao}
\end{equation}
\end{Rem}

We will need the following key estimate.

\begin{lem}
\label{finite-diff} For any $\delta \in (0,1)$ we have
\begin{equation*}
E\sup_{|\theta |\leq \delta }\int_{0}^{T-\delta }|u^{\alpha }(t+\theta
)-u^{\alpha }(t)|_{\mathbb{H}^{-4}}^{2}\leq C\delta .
\end{equation*}
\end{lem}

\begin{proof}
In what follows we set
\begin{equation*}
\frac{\partial }{\partial x_{i}}=\partial _{i},\text{ for any }i,
\end{equation*}%
and we rewrite the first equation in (\ref{secondgrade}) as follows (see
\cite{IFTIMIE} for the details)
\begin{equation}
\begin{split}
& \frac{\partial }{\partial t}(u^{\alpha }-\alpha \Delta u^{\alpha })-\nu
\Delta u^{\alpha }+u^{\alpha }.\nabla u^{\alpha }-\alpha \sum_{j,k}\partial
_{j}\partial _{k}(u_{j}^{\alpha }\partial _{k}u^{\alpha })+\alpha
\sum_{j,k}\partial _{j}(\partial _{k}u_{j}^{\alpha }\partial _{k}u^{\alpha })
\\
& =\alpha \sum_{j,k}\partial _{k}(\partial _{k}u_{j}^{\alpha }\nabla
u_{j}^{\alpha })-\nabla \mathfrak{P}^{\sharp }+F+G\frac{d\mathcal{W}}{dt},
\end{split}
\label{secondgrade3}
\end{equation}%
where
\begin{equation*}
\nabla \mathfrak{P}^{\sharp }=\frac{1}{2}\nabla (|u^{\alpha }|^{2}+\alpha
|\nabla u^{\alpha }|^{2})+\nabla \mathfrak{P}.
\end{equation*}%
We set $\Phi =\mathbb{P}(u^{\alpha }-\alpha \Delta u^{\alpha })$ where $%
\mathbb{P}$ is the Leray projector. We see from (\ref{secondgrade3}) that
\begin{equation*}
\begin{split}
& d\Phi +\{\nu Au^{\alpha }+\mathbb{P}(u^{\alpha }.\nabla u^{\alpha
})-\alpha \sum_{j,k}\mathbb{P}(\partial _{j}\partial _{k}(u_{j}^{\alpha
}\partial _{k}u^{\alpha }))+\alpha \sum_{j,k}\mathbb{P}(\partial
_{j}(\partial _{k}u_{j}^{\alpha }\partial _{k}u^{\alpha }))\}dt \\
& =\alpha \sum_{j,k}\mathbb{P}(\partial _{k}(\partial _{k}u_{j}^{\alpha
}\nabla u_{j}^{\alpha }))dt+Fdt+Gd\mathcal{W}.
\end{split}%
\end{equation*}%
This implies
\begin{equation*}
\begin{split}
\Phi (t+\theta )-\Phi (t)& =\int_{t}^{t+\theta }\{-\nu Au^{\alpha }+\alpha
\sum_{j,k}\left( \mathbb{P}(\partial _{j}\partial _{k}(u_{j}^{\alpha
}\partial _{k}u^{\alpha }))-\mathbb{P}(\partial _{j}(\partial
_{k}u_{j}^{\alpha }\partial _{k}u^{\alpha }))\right) \}ds \\
& -\int_{t}^{t+\theta }\mathbb{P}(u^{\alpha }.\nabla u^{\alpha
})ds+\int_{t}^{t+\theta }\{\alpha \sum_{j,k}\mathbb{P}(\partial
_{k}(\partial _{k}u_{j}^{\alpha }\nabla u_{j}^{\alpha
}))+F\}ds+\int_{t}^{t\theta }Gd\mathcal{W},
\end{split}%
\end{equation*}%
for any $\theta >0$. We infer from this that
\begin{equation*}
\begin{split}
|\Phi (t+\theta )-\Phi (t)|_{\mathbb{H}^{-4}}^{2}& \leq 2\left(
\int_{t}^{t+\theta }\left\{ +|u^{\alpha }.\nabla u^{\alpha }|_{\mathbb{H}%
^{-4}}+\alpha \sum_{j,k}|\partial _{j}\partial _{k}(u_{j}^{\alpha }\partial
_{k}u^{\alpha })|_{\mathbb{H}^{-4}}\right\} ds\right) ^{2} \\
& +4\left( \int_{t}^{t+\theta }\left\{ \alpha \sum_{j,k}[|\partial
_{j}(\partial _{k}u_{j}^{\alpha }\partial _{k}u^{\alpha })|_{\mathbb{H}%
^{-4}}+|\partial _{k}(\partial _{k}u_{j}^{\alpha }\nabla u_{j}^{\alpha })|_{%
\mathbb{H}^{-4}}]+|F|_{\mathbb{H}^{-4}}\right\} dst\right) ^{2} \\
& +2\int_{t}^{t+\theta }\nu |Au^{\alpha }|_{\mathbb{H}^{-4}}ds+2\left\vert
\int_{t}^{t+\theta }Gd\mathcal{W}\right\vert _{\mathbb{H}^{-4}}^{2},
\end{split}%
\end{equation*}%
which implies
\begin{equation*}
\begin{split}
|\Phi (t+\theta )-\Phi (t)|_{\mathbb{H}^{-4}}^{2}& \leq C\theta
\int_{t}^{t+\theta }\left\{ +|u^{\alpha }.\nabla u^{\alpha }|_{\mathbb{H}%
^{-4}}^{2}+\alpha ^{2}\sum_{j,k}|\partial _{j}\partial _{k}(u_{j}^{\alpha
}\partial _{k}u^{\alpha })|_{\mathbb{H}^{-4}}^{2}\right\} dt \\
& +C\theta \int_{t}^{t+\theta }\left\{ \alpha ^{2}\sum_{j,k}[|\partial
_{j}(\partial _{k}u_{j}^{\alpha }\partial _{k}u^{\alpha })|_{\mathbb{H}%
^{-4}}^{2}+|\partial _{k}(\partial _{k}u_{j}^{\alpha }\nabla u_{j}^{\alpha
})|_{\mathbb{H}^{-4}}^{2}]+|F|^{2}\right\} dt \\
& +C\theta \int_{t}^{t+\theta }\nu |Au^{\alpha }|_{\mathbb{H}%
^{-4}}^{2}ds+2\left\vert \int_{t}^{t+\theta }Gd\mathcal{W}\right\vert _{%
\mathbb{H}^{-4}}^{2}.
\end{split}%
\end{equation*}%
It is not hard to see that
\begin{equation}
|Au^{\alpha }|_{\mathbb{H}^{-4}}^{2}\leq C|u^{\alpha }|^{2}.  \label{i}
\end{equation}%
For $n=2$ Theorem \ref{product-formula} implies that
\begin{align}
|u^{\alpha }.\nabla u^{\alpha }|_{\mathbb{H}^{-4}}^{2}& \leq C|u^{\alpha
}|^{2}|\nabla u^{\alpha }|^{2},  \label{ii} \\
\alpha ^{2}|\partial _{j}\partial _{k}(u_{j}^{\alpha }\partial _{k}u^{\alpha
})|_{\mathbb{H}^{-4}}^{2}& \leq C|u_{j}^{\alpha }\partial _{k}u^{\alpha }|_{%
\mathbb{H}^{-2}}^{2}\leq C|u^{\alpha }|^{2}|\nabla u^{\alpha }|^{2}.
\label{iii}
\end{align}%
From the same theorem we have that
\begin{equation*}
|\partial _{k}u_{j}^{\alpha }\partial _{k}u^{\alpha }|_{\mathbb{H}^{-2}}\leq
C|\nabla u^{\alpha }|^{2}\quad \forall k,j,
\end{equation*}%
from which we derive that
\begin{equation}
\alpha ^{2}|\partial _{j}(\partial _{k}u_{j}^{\alpha }\partial _{k}u^{\alpha
})|_{\mathbb{H}^{-4}}^{2}\leq \alpha C\alpha |\nabla u^{\alpha }|^{2}|\nabla
u^{\alpha }|^{2}.  \label{iv}
\end{equation}%
A similar argument can be used to show that
\begin{equation}
\alpha ^{2}|\partial _{k}(\partial _{k}u_{j}^{\alpha }\nabla u_{j}^{\alpha
})|_{\mathbb{H}^{-4}}^{2}\leq \alpha C\alpha |\nabla u^{\alpha }|^{2}|\nabla
u^{\alpha }|^{2}.  \label{v}
\end{equation}%
The estimates (\ref{i})-(\ref{v}) along with (\ref{lem52}) allow us to write
\begin{equation*}
\begin{split}
E\int_{0}^{T-\delta }\sup_{0\leq \theta \leq \delta }|\Phi (t+\theta )-\Phi
(t)|_{\mathbb{H}^{-4}}^{2}dt& \leq C\delta ^{2}+C\delta +C\delta
E\int_{0}^{T-\delta }\int_{t}^{t+\theta }\alpha |\nabla u^{\alpha
}|^{2}|\nabla u^{\alpha }|^{2}dsdt \\
& +CE\int_{0}^{T-\theta }\sup_{0\leq \theta \leq \delta }\left\vert
\int_{t}^{t+\theta }Gd\mathcal{W}\right\vert _{\mathbb{H}^{-4}}^{2}dt.
\end{split}%
\end{equation*}%
But (\ref{lem52}) implies that
\begin{equation*}
E\sup_{0\leq t\leq T}\alpha ^{\frac{p}{2}}|\nabla u^{\alpha }(t)|^{p}+\nu
E\left( \int_{0}^{T}|\nabla u^{\alpha }(t)|^{2}dt\right) ^{\frac{p}{2}}\leq
C,\,\,2\leq p<\infty .
\end{equation*}%
From which we deduce that
\begin{equation*}
E\int_{0}^{T-\delta }\sup_{0\leq \theta \leq \delta }|\Phi (t+\theta )-\Phi
(t)|_{\mathbb{H}^{-4}}^{2}dt\leq C\delta ^{2}+C\delta +C\delta
+CE\int_{0}^{T-\delta }\sup_{0\leq \theta \leq \delta }\left\vert
\int_{t}^{t+\theta }Gd\mathcal{W}\right\vert _{\mathbb{H}^{-4}}^{2}dt.
\end{equation*}%
By making use of the Martingale inequality, the assumption on $G$ we obtain
that
\begin{equation*}
E\int_{0}^{T-\delta }\sup_{0\leq \theta \leq \delta }|\Phi (t+\theta )-\Phi
(t)|_{\mathbb{H}^{-4}}^{2}dt\leq C\delta .
\end{equation*}%
For almost all $(t,\omega )\in \lbrack 0,T]\times \Omega $ we have
\begin{equation*}
u^{\alpha }(t+\theta )-u^{\alpha }(t)=(I+\alpha A)^{-1}(\Phi (t+\theta
)-\Phi (t)),
\end{equation*}%
which implies that
\begin{equation*}
|u^{\alpha }(t+\theta )-u^{\alpha }(t)|_{\mathbb{H}^{\beta }}^{2}<|\Phi
(t+\theta )-\Phi (t)|_{\mathbb{H}^{\beta }}^{2},\forall \beta \in \mathbb{R}.
\end{equation*}%
Indeed for any $\phi \in \mathbb{H}^{\beta }(D)$ such that $\Div\phi =0$ and
$\int_{D}\phi (x)dx=0$ we have
\begin{equation*}
|\phi |_{\mathbb{H}^{\beta }}^{2}=\sum_{j=1}^{\infty }|\phi _{j}|^{2}\lambda
_{j}^{2\beta }<\sum_{j=1}^{\infty }(1+\alpha \lambda _{j})|\phi
_{j}|^{2}\lambda _{j}^{2\beta },
\end{equation*}%
that is,
\begin{equation*}
|\phi |_{\mathbb{H}^{\beta }}^{2}<|\phi +\alpha A\phi |_{\mathbb{H}^{\beta
}}^{2},
\end{equation*}%
where $\phi =\sum_{j=1}^{\infty }\phi _{j}e_{j}$, and $Ae_{j}=\lambda
_{j}e_{j},\,\,j=1,2,...$; the $e_{j}$-s are the eigenfunctions of the
operator $A$ and the $\lambda _{j}$-s are the corresponding eigenvalues. It
follows from this remark that
\begin{equation*}
E\int_{0}^{T-\delta }\sup_{0\leq \theta \leq \delta }|u^{\alpha }(t+\theta
)-u^{\alpha }(t)|_{\mathbb{H}^{-4}}^{2}dt\leq C\delta .
\end{equation*}%
A similar argument can be carried out to proving the same estimate for the
case $\theta <0$.
\end{proof}

\subsection{Compactness result and passage to the limit}

%%%%%%%%%%%%%%%%%%%%%%%%%%%%%%%%%%%%%%%%%%%%%%%%%%%%%%%%%%%%%%%%%%%%%%%%%%%%%%%%%%%%%%%%%%%%%%%%%%%%%%%
%%%%%%%%%%%%%%%%%%%%%%%%%%%%%%%%%%%%%%%%%%%%%%%%%%%%%%%%%%%%%%%%%%%%%%%%%%%%%%%%%%%%%%%%%%%%%%%%%%%%%%%
The following compactness result plays a crucial role in the proof of the
tightness of the probability measures generated by the sequence $%
(u^\alpha)_{\alpha\in [0,1)}$.

\begin{lem}
\label{lemma4}  Let $\mu_n$, $\nu_n$ two sequences of positive real numbers
which tend to zero as $n\rightarrow \infty$, the injection of
\begin{equation*}
D_{\nu_n,\mu_n}=\left\lbrace q\in L^\infty(0,T;\mathbb{H})\cap L^2(0,T;V);
\sup_{n}\frac{1}{\nu_n}\sup_{|\theta|\le \mu_n}\left(
\int_0^T|q(t+\theta)-q(t)|^2_{\mathbb{H}^{-4}}\right)^{1/2}<\infty\right%
\rbrace
\end{equation*}
in $L^2(0,T;\mathbb{H})$ is compact.
\end{lem}

The proof, which is similar to the analogous result in \cite{bensoussan2},
follows from the application of Lemmas \ref{comsimon}, \ref{lem5} and \ref%
{finite-diff}. The space $D_{\nu _{n},\mu _{n}}$ is a Banach space with the
norm
\begin{equation*}
||q||_{D_{\nu _{n},\mu _{n}}}=\text{ess}\sup_{0\leq t\leq T}|q(t)|+\left(
\int_{0}^{T}||q(t)||^{2}\right) ^{1/2}+\sup_{n}\frac{1}{\nu _{n}}%
\sup_{|\theta |\leq \mu _{n}}\left( \int_{0}^{T}|q(t+\theta )-q(t)|_{\mathbb{%
H}^{-4}}^{2}\right) ^{1/2}.
\end{equation*}%
Alongside $D_{\nu _{n},\mu _{n}}$, we also consider the space $X_{p,\nu
_{n},\mu _{n}}$, $1\leq p<\infty $, of random variables $\zeta $ endowed
with the norm
\begin{equation*}
\begin{split}
E||\zeta ||_{X_{p,\nu _{n},\mu _{n}}}& =E\text{ess}\sup_{0\leq t\leq
T}|\zeta (t)|^{p}+E\left( \int_{0}^{T}||\zeta (t)||^{2}\right) ^{p/2} \\
& +E\sup_{n}\frac{1}{\nu _{n}}\sup_{|\theta |\leq \mu _{n}}\left(
\int_{0}^{T}|\zeta (t+\theta )-\zeta (t)|_{\mathbb{H}^{-4}}^{2}\right)
^{1/2};
\end{split}%
\end{equation*}
$X_{p,\nu _{n},\mu _{n}}$ is a Banach space.

\noindent Combining \eqref{lem52vao} and the estimates in Lemma \ref%
{finite-diff} we have

\begin{prop}
\label{prop1} For any real number $p\in \lbrack 1,\infty )$ and for any
sequences $\nu _{n},\mu _{n}$ converging to $0$ such that the series $%
\sum_{n}\frac{\sqrt{\mu _{n}}}{\nu _{n}}$ converges, the sequence $%
(u^{\alpha })_{\alpha \in \lbrack 0,1)}$ is bounded uniformly in $\alpha $
in $X_{p,\nu _{n},\mu _{n}}$ for all $n$.
\end{prop}

Next we consider the space $\mathfrak{S}=C(0,T;\mathbb{R}^{m})\times
L^{2}(0,T;\mathbb{H})$ equipped with the Borel $\sigma $-algebra $\mathcal{B}%
(\mathfrak{S})$. For $\alpha \in \lbrack 0,1)$, let $\Phi _{\alpha }$ be the
measurable $\mathfrak{S}$-valued mapping defined on $(\Omega ,\mathcal{F},P)$
by
\begin{equation*}
\Phi _{\alpha }(\omega )=(\mathcal{W}(\omega ),u^{\alpha }(\omega )).
\end{equation*}%
For each $\alpha $ we introduce a probability measure $\Pi ^{\alpha }$ on $(%
\mathfrak{S};\mathcal{B}(\mathfrak{S}))$ defined by
\begin{equation*}
\Pi ^{\alpha }(S)=P(\Phi _{\alpha }^{-1}(S)),\text{ for any }S\in \mathcal{B}%
(\mathfrak{S}).
\end{equation*}

\begin{thm}
\label{thm4} The family of probability measures $\lbrace \Pi^\alpha: \alpha
\in [0,1)\rbrace$ is tight in $(\mathfrak{S};\mathcal{B}(\mathfrak{S}))$.
\end{thm}

\begin{proof}
For $\varepsilon>0$ we should find compact subsets
\begin{equation*}
\Sigma_\varepsilon \subset C(0,T;\mathbb{R}^m); Y_\varepsilon\subset L^2(0,T;%
\mathbb{H}),
\end{equation*}
such that
\begin{align}
P\left( \omega: \mathcal{W}(\omega,.)\notin \Sigma_\varepsilon\right)\le
\frac{\varepsilon}{2},  \label{47*} \\
P\left( \omega: u^\alpha(\omega,.)\notin Y_\varepsilon\right)\le \frac{%
\varepsilon}{2},  \label{47**}
\end{align}
for all $\alpha$.

The quest for $\Sigma _{\varepsilon }$ is made by taking into account some
facts about Wiener process such as the formula
\begin{equation}
E|\mathcal{W}(t)-\mathcal{W}(s)|^{2j}=(2j-1)!(t-s)^{j},j=1,2,....  \label{48}
\end{equation}%
For a constant $L_{\varepsilon }>0$ depending on $\varepsilon $ to be fixed
later and $n\in \mathbb{N}$, we consider the set
\begin{equation*}
\Sigma _{\varepsilon }=\{\mathcal{W}(.)\in C(0,T;\mathbb{R}^{m}):\sup
_{\substack{ t,s\in \lbrack 0,T] \\ |t-s|<\frac{1}{n^{6}}}}n|\mathcal{W}(s)-%
\mathcal{W}(t)|\leq L_{\varepsilon }\}.
\end{equation*}%
The set $\Sigma _{\varepsilon }$ is relatively compact in $C(0,T;\mathbb{R}%
^{m})$ by Arzela-Ascoli's theorem. Furthermore $\Sigma _{\varepsilon }$ is
closed in $C(0,T;\mathbb{R}^{m})$, therefore it is compact in $C(0,T;\mathbb{%
R}^{m})$. Making use of Markov's inequality
\begin{equation*}
P(\omega ;\zeta (\omega )\geq \beta )\leq \frac{1}{\beta ^{k}}E[|\zeta
(\omega )|^{k}],
\end{equation*}%
for any random variable $\zeta $ and real numbers $k$ we get

\begin{align*}
P\left( \omega :\mathcal{W}(\omega )\notin \Sigma _{\varepsilon }\right) &
\leq P\left[ \cup _{n}\left\{ \omega :\sup_{\substack{ t,s\in \lbrack 0,T]
\\ |t-s|<\frac{1}{n^{6}}}}|\mathcal{W}(s)-\mathcal{W}(t)|\geq \frac{%
L_{\varepsilon }}{n}\right\} \right] , \\
& \leq \sum_{n=1}^{\infty }\sum_{i=0}^{n^{6}-1}\left( \frac{n}{%
L_{\varepsilon }}\right) ^{4}E\sup_{\frac{iT}{n^{6}}\leq t\leq \frac{(i+1)T}{%
n^{6}}}|\mathcal{W}(t)-\mathcal{W}(iTn^{-6}|^{4}, \\
& \leq C\sum_{n=1}^{\infty }\sum_{i=0}^{n^{6}-1}\left( \frac{n}{%
L_{\varepsilon }}\right) ^{4}(Tn^{-6})^{2}n^{6}=\frac{C}{L_{\varepsilon }^{4}%
}\sum_{n=1}^{\infty }\frac{1}{n^{2}},
\end{align*}%
where we have used \eqref{48}. Since the right hand side of \eqref{48} is
independent of $\alpha $, then so is the constant $C$ in the above estimate.
We take $L_{\varepsilon }^{4}=\frac{1}{2C\varepsilon }\left(
\sum_{n=1}^{\infty }\frac{1}{n^{2}}\right) ^{-1}$ and get \eqref{47*}.

Next we choose $Y_\varepsilon$ as a ball of radius $M_\varepsilon$ in $%
D_{\nu_n,\mu_m}$ centered at 0 and with $\nu_n,\mu_n$ independent of $%
\varepsilon$, converging to 0 and such that the series $\sum_{n}\frac{\sqrt{%
\mu_n}}{\nu_n}$ converges, from Lemma \ref{lemma4},  $Y_\varepsilon$ is a
compact subset of $L^2(0,T;\mathbb{H})$. Furthermore, we have
\begin{align*}
P\left(\omega: u^\alpha(\omega)\notin Y_\varepsilon\right) \le&
P\left(\omega: ||u^\alpha||_{D_{\nu_n,\mu_m}}>M_\varepsilon \right) \\
\le& \frac{1}{M_\varepsilon}\left(E||u^\alpha||_{D_{\nu_n,\mu_m}}\right), \\
\le& \frac{1}{M_\varepsilon}\left( E||u^\alpha||_{X_{1,\nu_n,\mu_n}}\right),
\\
\le& \frac{C}{M_\varepsilon},
\end{align*}
where $C>0$ is independent of $\alpha$ (see Proposition \ref{prop1} for the
justification.)

Choosing $M_\varepsilon=2C\varepsilon^{-1}$, we get \eqref{47**}. From the
inequalities \eqref{47*}-\eqref{47**} we deduce that
\begin{equation*}
P\left(\omega:\mathcal{W}(\omega)\in \Sigma_\varepsilon; u^\alpha(\omega)\in
Y_\varepsilon\right)\ge 1-\varepsilon,
\end{equation*}
for all $\alpha \in [0,1)$. This proves that for all $\alpha \in [0,1)$
\begin{equation*}
\Pi^\alpha(\Sigma_\varepsilon\times Y_\varepsilon)\ge 1-\varepsilon,
\end{equation*}
from which we deduce the tightness of $\lbrace \Pi^\alpha: \alpha\in
[0,1)\rbrace$ in $(\mathfrak{S},\mathcal{B}(\mathfrak{S}))$.
\end{proof}

Prokhorov's compactness result enables us to extract from $\left( \Pi
^{\alpha }\right) $ a subsequence $\left( \Pi ^{\alpha _{j}}\right) $ such
that
\begin{equation*}
\Pi ^{\alpha _{j}}\text{ weakly converges to a probability measure }\Pi
\text{ on }\mathfrak{S}.
\end{equation*}%
Skorokhod's Theorem ensures the existence of a complete probability space $(%
\bar{\Omega},\bar{\mathcal{F}},\bar{P})$ and random variables $(\mathcal{W}%
^{\alpha _{j}},u^{\alpha _{j}})$ and $(\bar{\mathcal{W}},v)$ defined on $(%
\bar{\Omega},\bar{\mathcal{F}},\bar{P})$ with values in $\mathfrak{S}$ such
that
\begin{align}
& \text{The probability law of }(\mathcal{W}^{\alpha _{j}},u^{\alpha _{j}})%
\text{ is }\Pi ^{\alpha _{j}},  \label{Sko4} \\
& \text{The probability law of }(\bar{\mathcal{W}},v)\text{ is }\Pi ,
\label{Sko3} \\
& \mathcal{W}^{\alpha _{j}}\rightarrow \bar{\mathcal{W}}\text{ in }C(0,T;%
\mathbb{R}^{m})\,\,\bar{P}-\text{a.s.,}  \label{Skorohod1} \\
& u^{\alpha _{j}}\rightarrow v\text{ in }L^{2}(0,T;\mathbb{H})\,\,\bar{P}-%
\text{a.s..}  \label{Skorohod2}
\end{align}%
We let $\bar{\mathcal{F}}^{t}$ be the $\sigma $-algebra generated by $(\bar{%
\mathcal{W}}(s),v(s)),0\leq s\leq t$ and the null sets of $\bar{\mathcal{F}}$%
. We will show that $\bar{\mathcal{W}}$ is an $\bar{\mathcal{F}}^{t}$%
-adapted standard $\mathbb{R}^{m}$-valued Wiener process. To fix this, it is
sufficient to show that for any $0<t_{1}<t_{2}<\ldots <t_{m}=T$, the
increments process $\bar{\mathcal{W}}(t_{j})-\bar{\mathcal{W}}(t_{j-1}))$
are independent with respect to $\bar{\mathcal{F}}^{t_{j-1}}$, distributed
normally with mean $0$ and variance $t_{j}-t_{j-1}$. That is, to show that
for any $\lambda _{j}\in \mathbb{R}^{m}$ and $i^{2}=-1$
\begin{equation}
\bar{E}\exp \left( i\sum_{j=1}^{m}\lambda _{j}(\bar{\mathcal{W}}(t_{j})-\bar{%
\mathcal{W}}(t_{j-1}))\right) =\prod_{j=1}^{m}\exp \left( -\frac{1}{2}%
\lambda _{j}^{2}(t_{j}-t_{j-1})\right) .  \label{wiener}
\end{equation}%
The equation \eqref{wiener} will follow if we have
\begin{equation}
\bar{E}\left[ \exp \left( i\lambda (\bar{\mathcal{W}}(t+\theta )-\bar{%
\mathcal{W}}(t))\right) \right/ \bar{\mathcal{F}}^{t}]=\exp \left( -\frac{%
|\lambda |^{2}\theta }{2}\right) .  \label{expected}
\end{equation}%
We rely on the fact that for any random variables $X$ and $Y$ on any
probability space $(\bar{\Omega},\bar{\mathcal{F}},\bar{P})$ such that $X$
is $\bar{\mathcal{F}}$-measurable and $\bar{E}|Y|<\infty $, $\bar{E}%
|XY|<\infty $, we have
\begin{equation*}
\bar{E}(XY/\bar{\mathcal{F}})=X\bar{E}(Y/\bar{\mathcal{F}}),\,\,\,\,\,\bar{E}%
\bar{E}(Y/\bar{\mathcal{F}})=\bar{E}(Y),
\end{equation*}%
that is,
\begin{equation}
\bar{E}(XY)=\bar{E}(X\bar{E}(Y/\bar{\mathcal{F}})).  \label{expectation}
\end{equation}%
Now, let us consider an arbitrary bounded continuous functional $\vartheta
_{t}(\mathcal{W},v)$ on $\mathfrak{S}$ depending only on the values of $%
\mathcal{W}$ and $v$ on $(0,T)$. Owing to the independence of $\mathcal{W}(t)
$ to $\vartheta _{t}(\mathcal{W},v)$ and the fact that $\mathcal{W}$ is a
Wiener process, we have
\begin{align*}
& E\left[ \exp \left( i\lambda (\mathcal{W}(t+\theta )-\mathcal{W}%
(t))\right) \vartheta _{t}(\mathcal{W},v)\right]  \\
& =E\left[ \exp \left( i\lambda (\mathcal{W}(t+\theta )-\mathcal{W}%
(t))\right) \right] E\left[ \vartheta _{t}(\mathcal{W},v)\right]  \\
& =\exp \left( -\frac{|\lambda |^{2}\theta }{2}\right) E\left[ \vartheta
_{t}(\mathcal{W},v)\right] .
\end{align*}%
In view of \eqref{Sko4}-\eqref{Sko3}, this implies that
\begin{align*}
& \bar{{E}}\left[ \exp \left( i\lambda (\mathcal{W}^{\alpha _{j}}(t+\theta
)-W^{\alpha _{j}}(t))\right) \vartheta _{t}(\mathcal{W}^{\alpha _{j}},v)%
\right]  \\
& ={\bar{E}}\left[ \exp \left( i\lambda (\mathcal{W}^{\alpha _{j}}(t+\theta
)-\mathcal{W}^{\alpha _{j}}(t))\right) \right] {\bar{E}}\left[ \vartheta
_{t}(\mathcal{W}^{\alpha _{j}},v)\right]  \\
& =\exp \left( -\frac{|\lambda |^{2}\theta }{2}\right) {\bar{E}}\left[
\vartheta _{t}(\mathcal{W}^{\alpha _{j}},v)\right] .
\end{align*}%
Now, the convergence \eqref{Skorohod1} and the continuity of $\vartheta $
allow us to pass to the limit in this latter equation and obtain
\begin{equation*}
\bar{E}\left[ \exp \left( i\lambda (\bar{\mathcal{W}}(t+\theta )-\bar{%
\mathcal{W}}(t))\right) \vartheta _{t}(\bar{\mathcal{W}},v)\right] =\exp
\left( -\frac{|\lambda |^{2}\theta }{2}\right) \bar{E}\left[ \vartheta _{t}(%
\bar{\mathcal{W}},v)\right] ,
\end{equation*}%
which, in view of \eqref{expectation}, implies \eqref{expected}. The choice
of the above filtration implies then that $\bar{\mathcal{W}}$ is a $\bar{%
\mathcal{F}}^{t}$-standard $m$-dimensional Wiener process.

By a similar method as used in \cite{bensoussan} (see also \cite{PAUL1}), we
can prove the following result.

\begin{thm}
\label{INTTHM} For any $j\ge1$, $\phi\in\mathcal{V}$, for all $t\in[0,T]$
the following holds almost surely
\begin{equation}
\begin{split}
(u^{\alpha_{j}},\phi)_{\mathbb{V}}+\int_{0}^{t}\{(\nu
Au^{\alpha_{j}}+B(u^{\alpha_{j}},u^{\alpha_{j}}),\phi)\}dt=(u_{0},\phi)_{%
\mathbb{V}}+\int_{0}^{t}(R(u^{\alpha_{j}})+F(u^{\alpha_{j}}),\phi)dt \\
+\int_{0}^{t}(G,\phi)d\mathcal{W}^{\alpha_{j}},
\end{split}
\label{thm8}
\end{equation}
where
\begin{align*}
B(u^{\alpha_{j}},u^{\alpha_{j}}) & =\mathbb{P}(u^{\alpha_{j}}.\nabla
u^{\alpha_{j}}), \\
R(u^{\alpha_{j}}) & =\alpha\sum_{i,k}\mathbb{P}\left(\partial_{i}%
\partial_{k}(u_{i}^{\alpha_{j}}\partial_{k}u^{\alpha_{j}})+\partial_{i}(%
\partial_{k}u_{i}^{\alpha_{j}}\partial_{k}u^{\alpha_{j}})-\partial_{k}(%
\partial_{k}u_{i}^{\alpha_{j}}\nabla u_{i}^{\alpha_{j}})\right)
\end{align*}
\end{thm}

To back the main theorem we have to pass to the limit in the equation (\ref%
{thm8}). Since $u^{\alpha _{j}}$ satisfies (\ref{thm8}) then $u^{\alpha _{j}}
$ satisfies the estimates in Lemma \ref{lem5}. Consequently, we can extract
from $(u^{\alpha _{j}})$ a subsequence denoted by the same symbol such that
\begin{align}
u^{\alpha _{j}}& \rightharpoonup v\text{ weak-}\ast \text{ in }L^{2}(\bar{%
\Omega},\bar{\mathcal{F}},\bar{P};L^{\infty }(0,T;\mathbb{H})),  \notag \\
u^{\alpha _{j}}& \rightharpoonup v\text{ weakly in }L^{2}(\bar{\Omega},\bar{%
\mathcal{F}},\bar{P};L^{2}(0,T;\mathbb{V})).  \label{conv1}
\end{align}%
We derive from (\ref{iii})-(\ref{v}) that
\begin{equation*}
R(u^{\alpha _{j}})\rightarrow 0\text{ in }L^{2}(\bar{\Omega},\bar{\mathcal{F}%
},\bar{P};L^{2}(0,T;\mathbb{H}^{-4})).
\end{equation*}%
Thus
\begin{equation*}
(R(u^{\alpha _{j}}),\phi )\rightarrow 0\text{ in }L^{2}(\bar{\Omega},\bar{%
\mathcal{F}},\bar{P};L^{2}(0,T)),\forall \phi \in \mathcal{V}.
\end{equation*}%
Since $A$ is linear and strongly continuous then owing to (\ref{conv1}) we
have
\begin{equation*}
Au^{\alpha _{j}}\rightharpoonup Av\text{ weakly in }L^{2}(\bar{\Omega},\bar{%
\mathcal{F}},\bar{P};L^{2}(0,T;\mathbb{H}^{-1})).
\end{equation*}%
Hence
\begin{equation*}
<Au^{\alpha _{j}},\phi >\rightharpoonup <Av,\phi >\text{ weakly in }L^{2}(%
\bar{\Omega},\bar{\mathcal{F}},\bar{P};L^{2}(0,T))\text{ for any }\phi \in
\mathcal{V}.
\end{equation*}%
We derive from (\ref{Skorohod2}), the estimate (\ref{lem52}) of Lemma \ref%
{lem5} and Vitali's Theorem that
\begin{equation}
u^{\alpha _{j}}\rightarrow v\text{ strongly in }L^{2}(\bar{\Omega},\bar{%
\mathcal{F}},\bar{P};L^{2}(0,T;\mathbb{H})).  \label{CONV2}
\end{equation}%
For any element $\zeta \in L^{\infty }(\bar{\Omega}\times \lbrack 0,T],d\bar{%
P}\otimes dt)$ and for any $\phi \in \mathcal{V}$ we have
\begin{equation*}
E\int_{0}^{T}<B(u^{\alpha _{j}},u^{\alpha _{j}}),\zeta \phi
>dt=-\sum_{i,k}E\int_{D\times \lbrack 0,T]}u_{i}^{\alpha _{j}}\partial
_{i}\phi _{k}\zeta u_{k}^{\alpha _{j}}dx\otimes dt.
\end{equation*}%
Owing to (\ref{CONV2})
\begin{equation*}
\zeta \partial _{i}\phi _{k}u_{k}^{\alpha _{j}}\rightarrow \zeta \partial
_{i}\phi _{k}v_{k}\text{ strongly in }L^{2}(\bar{\Omega},\bar{\mathcal{F}},%
\bar{P};L^{2}(0,T;\mathbb{H})).
\end{equation*}%
This and (\ref{CONV2}) again imply that
\begin{equation*}
\begin{split}
-\sum_{i,k}E\int_{D\times \lbrack 0,T]}u_{i}^{\alpha _{j}}\partial _{i}\phi
_{k}\zeta u_{k}^{\alpha _{j}}dx\otimes dt& \rightarrow
-\sum_{i,k}E\int_{D\times \lbrack 0,T]}v_{i}\partial _{i}\phi _{k}\zeta
v_{k}dx\otimes dt \\
& =E\int_{0}^{T}<B(v,v),\zeta \phi >dt.
\end{split}%
\end{equation*}%
That is
\begin{equation*}
<u^{\alpha _{j}}.\nabla u^{\alpha _{j}},\phi >\rightharpoonup <v.\nabla
v,\phi >\text{ weakly in }L^{2}(\bar{\Omega},\bar{\mathcal{F}},\bar{P}%
;L^{\infty }(0,T))\text{ for any }\phi \in \mathcal{V}.
\end{equation*}%
We readily have that
\begin{equation*}
\int_{0}^{t}(G,\phi )d\mathcal{W}^{\alpha _{j}}\rightharpoonup
\int_{0}^{t}(G,\phi )d\bar{\mathcal{W}}\text{ weakly in }L^{2}(\bar{\Omega},%
\bar{\mathcal{F}},\bar{P};L^{\infty }(0,T))\text{ for any }\phi \in \mathcal{%
V}.
\end{equation*}%
We have that
\begin{align*}
(u^{\alpha _{j}},\phi )_{\mathbb{V}}& =(u^{\alpha _{j}}-\alpha _{j}\Delta
u^{\alpha _{j}},\phi ) \\
& =(u^{\alpha _{j}},\phi )+\alpha _{j}((u^{\alpha _{j}},\phi )).
\end{align*}%
This implies that
\begin{equation*}
(u^{\alpha _{j}}-\alpha _{j}\Delta u^{\alpha _{j}}-v,\phi )=(u^{\alpha
_{j}}-v,\phi )+\alpha _{j}((u^{\alpha _{j}},\phi )).
\end{equation*}%
It follows from Lemma \ref{lem5} and (\ref{CONV2}) that
\begin{equation*}
(u^{\alpha _{j}}-\alpha _{j}\Delta u^{\alpha _{j}}-v,\phi )\rightarrow 0%
\text{ strongly in }L^{2}(\bar{\Omega},\bar{\mathcal{F}},\bar{P};L^{\infty
}(0,T))\text{ for any }\phi \in \mathcal{V}.
\end{equation*}%
Using all these convergences we can derive from (\ref{thm8}) that the
following holds almost surely
\begin{equation*}
(v,\phi )+\nu \int_{0}^{t}\{((v,\phi ))+(\mathbb{P}(v.\nabla v),\phi
)\}ds=(u_{0},\phi )+\int_{0}^{t}(F(v),\phi )ds+\int_{0}^{t}(G,\phi )d\bar{%
\mathcal{W}},
\end{equation*}%
for any $\phi \in \mathcal{V}$ and $t\in \lbrack 0,T]$. That is the system $(%
\bar{\Omega},\bar{\mathcal{F}},\bar{\mathcal{F}}^{t},\bar{P});(\bar{\mathcal{%
W}},v)$ is a weak solution of the stochastic Navier-Stokes equations.\newline
Owing to the estimates
\begin{equation*}
\bar{E}\sup_{0\leq s\leq T}|v(s)|^{\frac{p}{2}}+\nu \bar{E}\left(
\int_{0}^{T}||v(s)||^{2}ds\right) ^{\frac{P}{2}}<\infty ,
\end{equation*}%
the $\mathbb{H}$-valued process $v(.)$ has almost surely a weak-continuous
modification. This ends the proof of Theorem \ref{Mainthm}.

\section{Proof of Theorem \protect\ref{CONVPROB}}

The main ingredients of the proof of Theorem \ref{CONVPROB} are the pathwise
uniqueness for the two-dimensional stochastic Navier-Stokes equations and
the following lemma whose proof can be found in \cite{GYONGY}.

\begin{lem}
\label{CONVPROB1} Let $X$ be a Polish space. A sequence of a X-valued random
variables $\{x_{n};n\ge0\}$ converges in probability if and only if for
every subsequence of joint probability laws, $\{\nu_{n_{k},m_{k}};k\ge0\}$,
there exists a further subsequence which converges weakly to a probability
measure $\nu$ such that
\begin{equation*}
\nu\left(\{(x,y)\in X\times X;x=y\}\right)=1.
\end{equation*}
\end{lem}

Now, let
\begin{equation*}
\mathfrak{S}=L^{2}(0,T;\mathbb{H})\times C(0,T;\mathbb{R}^{m}),
\end{equation*}%
\begin{equation*}
\mathfrak{S}^{\mathbb{H}}=L^{2}(0,T;\mathbb{H}),\quad \mathfrak{S}^{\mathcal{%
W}}=C(0,T:\mathbb{R}^{m}),
\end{equation*}%
and
\begin{equation*}
\mathfrak{S}^{\mathbb{H},\mathbb{H}}=L^{2}(0,T;\mathbb{H})\times L^{2}(0,T;%
\mathbb{H}).
\end{equation*}%
For any $S\in \mathcal{B}(\mathfrak{S}^{\mathbb{H}})$ we set $\Pi ^{\alpha
}(S)=P(u^{\alpha }\in S)$, and $\Pi _{\mathcal{W}}^{\alpha }(S)=P(\mathcal{W}%
\in S)$ for $S\in \mathcal{B}(\mathfrak{S}^{\mathcal{W}})$. Next, we define
the joint probability laws
\begin{align*}
\Pi ^{\alpha ,\beta }& =\Pi ^{\alpha }\times \Pi ^{\beta }, \\
\nu ^{\alpha ,\beta }& =\Pi ^{\alpha }\times \Pi ^{\beta }\times \Pi _{%
\mathcal{W}}^{\alpha }.
\end{align*}%
The following tightness property holds.
\begin{lem}\label{lem5}
The collection $\nu ^{\alpha ,\beta }$
(and hence any subsequence $\{\nu ^{\alpha _{j},\beta _{j}}\}$) is tight on $%
\mathfrak{S}^{\mathbb{H},\mathbb{H}}\times \mathfrak{S}^{\mathcal{W}}$.
\end{lem}
\begin{proof}
 The proof is very similar to Theorem \ref{thm4}. For any $\varepsilon>0$ we choose the sets $\Sigma_\varepsilon, Y_\varepsilon$ exactly as in the proof of Theorem \ref{thm4}
with appropriate modification on the constants $M_\varepsilon, L_\varepsilon$ so that $\Pi^\alpha(Y_\varepsilon)\ge 1-\frac{\varepsilon}{4}$ and $\Pi_\mathcal{W}(\Sigma_\varepsilon)\ge 1-\frac{\varepsilon}{2}$ for every
$\alpha \in (0,1)$. Now let us take $K_\varepsilon=Y_\varepsilon\times Y_\varepsilon\times \Sigma_\varepsilon$ which is a compact in $\mathfrak{S}$; it is not difficult to see that
 $\{\nu^{\alpha, \beta}(K_\varepsilon)\ge (1-\frac{\varepsilon}{4})^2(1-\frac{\varepsilon}{2}) \ge 1-\varepsilon$ for all $\alpha, \beta$. This completes the proof of the lemma.
\end{proof}
Lemma \ref{lem5} implies that there exists a subsequence from $\{\nu ^{\alpha _{j},\beta
_{j}}\}$ still denoted by $\{\nu ^{\alpha _{j},\beta _{j}}\}$ which
converges to a probability measure $\nu $. By Skorokhod's theorem there
exists a probability space $(\bar{\Omega},\bar{\mathcal{F}},\bar{P})$ on
which a sequence $(u^{\alpha _{j}},v^{\beta _{j}},\mathcal{W}^{j})$ is
defined and converges almost surely in $\mathfrak{S}^{\mathbb{H},\mathbb{H}%
}\times \mathfrak{S}^{\mathcal{W}}$ to a couple of random variables $(u,v,W)$%
. Furthermore, we have
\begin{align*}
Law(u^{\alpha _{j}},v^{\beta _{j}},\mathcal{W}^{j})& =\nu ^{\alpha
_{j},\beta _{j}}, \\
Law(u,v,W)& =\nu .
\end{align*}%
Now let $Z_{j}^{u}=(u^{\alpha _{j}},\mathcal{W}^{j})$, $Z_{j}^{v}=(v^{\beta
_{j}},\mathcal{W}^{j})$, $Z^{u}=(u,W)$ and $Z^{v}=(v,W)$ We can infer from
the above argument that $\left( \Pi ^{\alpha _{j},\beta _{j}}\right) $
converges to a measure $\Pi $ such that
\begin{equation*}
\Pi (\cdot )=\bar{P}((u,v)\in \cdot ).
\end{equation*}%
As above we can show that $Z_{j}^{u}$ and $Z_{j}^{v}$ satisfy Theorem \ref%
{INTTHM} and that $Z^{u}$ and $Z^{v}$ satisfy the stochastic Navier-Stokes
equation (see \eqref{defweaknc}) on the same stochastic system $(\bar{\Omega}%
,\bar{\mathcal{F}},\bar{P}),\bar{\mathcal{F}}^{t},W$. Since we are in
two-dimensional case then we see that $u(0)=v(0)$ almost surely and $u=v$ in
$L^{2}(0,T;\mathbb{H})$. Therefore
\begin{equation*}
\Pi \left( \{(x,y)\in \mathfrak{S}^{\mathbb{H},\mathbb{H}};x=y\}\right) =%
\bar{P}\left( u=v\text{ in }L^{2}(0,T;\mathbb{H})\right) =1.
\end{equation*}%
This fact together with Lemma \ref{CONVPROB1} imply that the original
sequence $\left( u^{\alpha }\right) $ defined on the original probability
space $(\Omega ,\mathcal{F},P)$ converges in probability to an element $v$
in the topology of $\mathfrak{S}^{\mathbb{H}}$.  By a passage to the limits argument as in the
previous section it is not difficult to show that $v$ is the unique solution of the stochastic Navier-Stokes Equations (on the original probability system
$(\Omega ,\mathcal{F},\mathbb{P}),\mathcal{F}^t, W$). This ends
the proof of Theorem \ref{CONVPROB}.

\section*{Acknowledgment}

The research of the authors is supported by the University of Pretoria and
the National Research Foundation South Africa. The first author is also very
grateful to the support he received from The Abdus Salam International Center
for Theoretical Physics.

\end{document}